\documentclass[english,12pt,oneside]{amsproc}
\usepackage[english]{babel}
\usepackage{a4wide}
\usepackage{amsthm}
\usepackage{graphics}
\usepackage{amsfonts, amssymb, amscd, amsmath}
\usepackage{latexsym}
\usepackage[matrix,arrow,curve]{xy}
\usepackage{mathabx}
\usepackage{color}
\usepackage{pbox}
\usepackage{tikz}
\usetikzlibrary{matrix,decorations.pathreplacing,positioning}
\usepackage{hyperref}

\DeclareMathOperator{\Cone}{Cone} 
\DeclareMathOperator{\Ker}{Ker}

\DeclareMathOperator{\Hom}{Hom} \DeclareMathOperator{\rk}{rk}

\DeclareMathOperator{\GL}{GL}
\DeclareMathOperator{\odd}{odd}
\DeclareMathOperator{\Flats}{Flats}
\DeclareMathOperator{\Gr}{Gr}

\newcommand{\Zo}{\mathbb{Z}}
\newcommand{\Ro}{\mathbb{R}}
\newcommand{\Rg}{\mathbb{R}_{\geqslant 0}}
\newcommand{\Co}{\mathbb{C}}
\newcommand{\Qo}{\mathbb{Q}}

\newcommand{\B}{\mathbf{B}}
\newcommand{\Sp}{\mathbf{Sp}}

\newcommand{\ca}[1]{\mathcal{#1}}

\newcommand{\dd}{\partial}

\newcommand{\F}{\mathcal{F}}

\newcommand{\Nq}{N_\Qo}

\newcommand{\CP}{\mathbb{C}P}

\newcounter{stmcounter}[section]

\numberwithin{equation}{section}

\theoremstyle{plain}
\newtheorem{cor}[stmcounter]{Corollary}

\newtheorem{thm}[stmcounter]{Theorem}

\newtheorem{prop}[stmcounter]{Proposition}
\newtheorem{lem}[stmcounter]{Lemma}

\theoremstyle{definition}
\newtheorem{defin}[stmcounter]{Definition}

\theoremstyle{remark}
\newtheorem{ex}[stmcounter]{Example}
\newtheorem{rem}[stmcounter]{Remark}
\newtheorem{con}[stmcounter]{Construction}

\begin{document}

\title{Toric orbit spaces which are manifolds}

\author{Anton Ayzenberg}
\address{Faculty of computer science, Higher School of Economics
}
\email{ayzenberga@gmail.com}

\author{Vladimir Gorchakov}
\address{Faculty of computer science, Higher School of Economics}
\email{vyugorchakov@edu.hse.ru}

\date{\today}
\thanks{The article was prepared within the framework of the HSE University Basic Research Program}



\begin{abstract}
We characterize the actions of compact tori on smooth manifolds for which the orbit space is a topological manifold (either closed or with boundary). For closed manifolds the result was originally proved by Styrt in 2009. We give a new proof for closed manifolds which is also applicable to manifolds with boundary. In our arguments we use the result of Provan and Billera who characterized matroid complexes which are pseudomanifolds. We study the combinatorial structure of torus actions whose orbit spaces are manifolds. In two appendix sections we give an overview of two theories related to our work. The first one is the combinatorial theory of Leontief substitution systems from mathematical economics. The second one is the topological Kaluza--Klein model of Dirac's monopole studied by Atiyah. The aim of these sections is to draw some bridges between disciplines and motivate further studies in toric topology.
\end{abstract}

\maketitle

\section{Introduction}\label{secIntro}

The classical object of toric topology are the actions of compact tori $T$ on smooth manifolds $X$ which have topological orbit space $X/T$ isomorphic to simple polytopes. The examples include smooth projective toric varieties and their topological generalizations: quasitoric manifolds and torus manifolds, and moment-angle manifolds. All these manifolds are examples of torus actions of complexity zero. In general, it is known that orbit spaces of actions of complexity zero are manifolds with corners.

Buchstaber and Terzi\'{c}~\cite{BT} initiated the study of orbit spaces of torus actions of positive complexity. In particular, they proved that the canonical torus action on a Grassmann manifold $\Gr_{4,2}$ of $2$-planes in $\Co^4$ has orbit space homeomorphic to a sphere and a similar result for the canonical torus action of the variety of full flags in $\Co^3$. In~\cite{AyzCompl}, the first author proved a local statement that, under certain assumption, the orbit space of a complexity one action in general position is a closed topological manifold. This result was later extended by Cherepanov in~\cite{Cher}: the orbit spaces of complexity one actions in non-general position are manifolds with corners.

A natural question arises: describe torus actions on manifolds which have orbit spaces either manifolds or manifolds with boundary. This question stays in parallel to the seminal work of Vinberg~\cite{Vin}, and the later works of Mikhailova and Lange~\cite{Mikh,Lange16,Lange,LangMikh} who classified all finite group actions whose orbit spaces are closed manifolds. Using the slice theorem, the questions of this sort reduce to linear representations of the corresponding group. We prove the following results about linear representations.

\begin{thm}\label{thmIntroOpen}
Assume that the orbit space of a representation of a compact torus $T$ on a real vector space $V$ is homeomorphic to $\Ro^m$ for some $m$. Then the representation is weakly equivalent to a product of complexity one representations in general position and, probably, a trivial representation.
\end{thm}

\begin{thm}\label{thmIntroBoundary}
Assume that the orbit space of a representation of a compact torus $T$ on a real vector space $V$ is homeomorphic to a half-space $\Rg\times\Ro^{m-1}$. Then the representation is weakly equivalent to a complexity zero representation, probably multiplied by a product of complexity one representations in general position and a trivial representation.
\end{thm}

After we obtained these results, we found that Theorem~\ref{thmIntroOpen} was already proved by Styrt~\cite{Styrt} in a greater generality. The arguments of his proof are essentially similar to ours. However, in our case one essential step of the proof is simplified since we refer to a known result of Provan and Billera~\cite{BilProv} from matroid theory. The same result is applied in the proof of Theorem~\ref{thmIntroBoundary}. The cited result about matroids (see Proposition~\ref{propBillProv}) was originally motivated by the combinatorial study of Leontief substitution systems. For this reason we call the products of representations of complexities zero and one, which appear in the statements of the theorems, \emph{Leontief representations}, see Definition~\ref{definLeontRepres}.

The paper has the following structure. Section~\ref{secDefinitions} contains all basic definitions, examples, and a rigorous statement of the result. Theorems~\ref{thmIntroOpen} and~\ref{thmIntroBoundary} are proved in Section~\ref{secProofs}. In Sections~\ref{secTorusActions} and~\ref{secLowerFaces} we recall basic notions related to torus actions on smooth manifolds, and introduce Leontief torus actions, which are basically the actions whose orbit space is a manifold (with or without boundary). We  study the combinatorial structure of Leontief torus actions and relate these results to the works~\cite{AyzCompl,Cher} on torus actions of complexity one.

There are two appendix sections with a general exposition of two topics related to our work. Appendix~\ref{secLeontief} contains a brief overview of Leontief substitution systems and explains the (somewhat equivocal) choice of the term Leontief representation. Appendix~\ref{secMonopoles} is devoted to the topological aspects of Kaluza--Klein model with Dirac's magnetic monopoles. Essentially, this model is a very particular example of a torus action of complexity one in general position. The important property of this model is that the orbit space of an action is an open manifold. Keeping this in mind, one can consider Leontief torus actions as the most general higher-dimensional analogues of the topological Kaluza--Klein theory. This observation may become a good motivation for the further study of such torus actions.

\section{Definitions and results}\label{secDefinitions}

In the paper $T^k$ denotes the compact $k$-dimensional torus considered as a Lie group. The lattice $N=\Hom(T^k,S^1)\cong \Zo^k$ is called \emph{the weight lattice}, and its dual lattice $N^*=\Hom(S^1,T^k)$ is called \emph{the lattice of 1-dimensional subgroups}.

Consider a representation of $T=T^k$ on $V=\Ro^{m}$. It decomposes into a direct sum of irreducible representations. An irreducible representation of the abelian group $T$ has real dimension 1 or 2. One-dimensional representations are trivial (no-actions). A 2-dimensional real representation $V(\alpha)$ is determined by a nonzero weight $\alpha\in N$, so that
\[
V(\alpha)\cong \Co\cong\Ro^2,\quad tz=\alpha(t)\cdot z.
\]
Since neither a complex structure nor an orientation is fixed on $\Ro^2$, the weight $\alpha$ is determined uniquely up to sign. Therefore, an arbitrary representation $V$ of a torus decomposes into the sum
\begin{equation}\label{eqDecomposition}
V\cong V(\alpha_1)\oplus\ldots\oplus V(\alpha_r)\oplus \Ro^{m-2r},
\end{equation}
where the torus action is trivial on $\Ro^{m-2r}$, and $\alpha_1,\ldots,\alpha_r\in N$ is a collection of nonzero vectors of the weight lattice $N$, defined up to sign.

In the following we consider weights as rational vectors in the vector space $\Nq=N\otimes\Qo\cong\Qo^k$. Moreover, since weights are defined up to sign, we can treat them as rational lines (one-dimensional vector subspaces). Although a rational line contains infinitely many nonzero integral vectors, the choice of a representative is nonessential for our arguments. It follows that any torus representation is completely characterized by a multiset $\alpha=\{\alpha_1,\ldots,\alpha_r\}$ of vectors (or lines) in $\Nq$.

Notice that the summand $\Ro^{m-2r}$ in~\eqref{eqDecomposition} is the fixed point subspace of the representation. So far, the representation has isolated fixed point if and only if $m=2r$.

In the following it is assumed that torus representations are effective, which means their weights span the vector space $\Nq$.

\begin{defin}\label{definComplexity}
Consider a representation of $T=T^k$ on $V=\Ro^m$ with the weight system $\alpha=\{\alpha_1,\ldots,\alpha_r\}$. The number $r-k=|\alpha|-\rk\alpha$ is called \emph{the complexity of the representation}.
\end{defin}

Notice that complexity is always a nonnegative number and does not depend on the dimension of the trivial summand of the representation. If a representation $V$ has isolated fixed point, its complexity equals $\frac{1}{2}\dim V-\dim T$.

Change of coordinates in a torus motivates the definition of weakly equivalent representations.

\begin{defin}
Two representations $V$ and $W$ of $T$ are called weakly equivalent, if there is an automorphism $\psi\colon T\to T$ and an isomorphism $g\colon V\to W$ such that $g(tv)=\psi(t)g(v)$.
\end{defin}

\begin{ex}\label{exCompl0}
A representation of complexity 0 takes the form
\[
V(\alpha_1)\oplus\cdots\oplus V(\alpha_n)\oplus \Ro^{m-2n},
\]
where $\alpha=\{\alpha_1,\ldots,\alpha_n\}\subset N=\Hom(T^n,T^1)$ is a rational basis of $\Nq\cong\Qo^n$. Using Smith normal form over $\Zo$, we can change coordinates in $T^n$ (or equivalently in $N$). Therefore, up to weak equivalence, we have $\alpha_i=d_ie_i$, where $\{e_1,\ldots,e_n\}$ is the basis of the lattice $N$, and $d_i$'s are nonzero integers satisfying $d_1\mid d_2\mid\cdots\mid d_n$. Assuming there is no trivial component, a complexity zero action takes the form
\[
(t_1,\ldots,t_n)(z_1,\ldots,z_n)=(t_1^{d_1}z_1,\ldots,t_n^{d_n}z_n),
\]
for $z_i\in\Co$. If $d_i=1$ for any $i$, the representation is called \emph{standard}. This class of representations is well studied and widely used in toric topology.

However, even for general $d_i$'s, the orbit space of the complexity zero representation (without trivial component) is a nonnegative cone $\Rg^n$. As a topological space it is homeomorphic to the halfspace $\Rg\times\Ro^{n-1}$.
\end{ex}

\begin{defin}\label{definComplOneGenPos}
A representation of $T=T^{n-1}$ on $V\cong\Ro^{2n}$ is called \emph{a complexity one representation in general position} if its trivial part vanishes, and any $n-1$ of the weights $\alpha=\{\alpha_1,\ldots,\alpha_n\}\subset\Nq\cong\Qo^{n-1}$ are linearly independent over $\Qo$.
\end{defin}

For a complexity one representation in general position, the collection $\alpha=\{\alpha_1,\ldots,\alpha_n\}\subset N$ determines a group homomorphism
\[
A=\prod_{i=1}^n\alpha_i\colon T^{n-1}\to T^n.
\]
Since $\alpha$ spans $N_\Qo$, the kernel $\Ker A$ is a finite abelian group. The image of $A$ is a codimension 1 toric subgroup $\{(t_1,\ldots,t_n)\in T^n\mid \prod_{i=1}^{n}t_i^{c_i}=1\}$ where $c_1\alpha_1+\cdots+c_n\alpha_n=0$ is a unique (up to multiplier) linear relation on $n$ vectors $\alpha_i$ in $\Nq\cong\Qo^{n-1}$, and $c_i$'s don't have a nontrivial common divisor. The condition that any $n-1$ of $\alpha_i$'s are independent is equivalent to $c_i\neq 0$ for any $i$.

Since $\alpha_i$ are only defined up to sign, we can assume that $c_i$'s are natural numbers. The orbit space for the original representation naturally coincides with that of the image of $A$, which implies the following observation.

\begin{lem}\label{lemTechnical}
The orbit space of a complexity one representation of $T^{n-1}$ in general position on $V\cong\Ro^{2n}$ is homeomorphic to the orbit space of the action of the subgroup $H=\left\{(t_1,\ldots,t_n)\in T^n\mid \prod_{i=1}^nt_i^{c_i}=1\right\}$, where $c_i>0$, induced by the standard action of $T^n$ on $\Co^n\cong \Ro^{2n}$.
\end{lem}

Notice that the stabilizer subgroups of the original action and those of $H$ do not necessarily coincide: these depend on the finite group $\Ker A$ described above. For the orbit space, however, one can use the particular model action of $H$ to prove the following result.


\begin{prop}[{\cite[Lm.2.11]{AyzCompl} or~\cite[Thm.3.6]{Styrt}}]\label{propLocalQuotient}
For a representation of $T^{n-1}$ on $\Co^n$ of complexity one in general position we have a homeomorphism $\Co^n/T^{n-1}\cong \Ro^{n+1}$.
\end{prop}

\begin{defin}\label{definLeontRepres}
Consider a collection of complexity one representations in general position $T^{n_i-1}\to\GL(\Co^{n_i})$, for $i\in\{1,\ldots,s\}$, a complexity zero representation $T^d\to\GL(\Co^d)$, and a trivial representation on $\Ro^l$. Then the product representation of $T^d\times\prod_{i=1}^{s}T^{n_i-1}$ on $V=\Ro^l\times\Co^d\times\prod_{i=1}^{s}\Co^{n_i}$ is called \emph{a Leontief representation}. It is called \emph{totally Leontief} if $d=0$.
\end{defin}

The reason for the chosen name of the term is explained in Appendix~\ref{secLeontief}. For a totally Leontief representation we have
\[
V/T=\Ro^l\times\prod_{i=1}^{s}\Co^{n_i}/T^{n_i-1}=\Ro^l\times\prod_{i=1}^{s}\Ro^{n_i+1},
\]
so the orbit space is a topological manifold. 
Similarly, the orbit space of any non-totally Leontief representation is a half-space, that is a manifold with boundary.

Theorems~\ref{thmIntroOpen} and~\ref{thmIntroBoundary} stated in the introduction assert that Leontief representations provide an exhaustive list of representations whose orbit spaces are manifolds (either open or bounded). The following is a reformulation of these theorems.

\begin{thm}\label{thmAll}\mbox{}
\begin{enumerate}
  \item Assume that the orbit space $V/T$ of a representation $T\to\GL(V)$ is homeomorphic to $\Ro^l$. Then the representation is weakly equivalent to a totally Leontief representation.
  \item Assume that the orbit space $V/T$ of a representation $T\to\GL(V)$ is homeomorphic to a half-space $\Rg\times\Ro^{l-1}$. Then the representation is weakly equivalent to a non-totally Leontief representation.
\end{enumerate}
\end{thm}

The first result was proved in~\cite[Thm.1.3]{Styrt} in a much bigger generality: Styrt characterized all representations of $G\subset \GL(V)$ with orbit spaces homeomorphic to $\Ro^d$, under the assumption that the connected component of $G$ is a compact torus. Our proof of the first part of Theorem~\ref{thmAll} essentially follows the lines of the proof in~\cite{Styrt} for the particular case of $G=T$, however we simplify the argument by referring to some known results about matroids. A similar technique is applied to prove item 2 of Theorem~\ref{thmAll}.

\section{Proofs}\label{secProofs}

Recall that \emph{a(n abstract) simplicial complex} on a vertex set $A$ is a collection $K\subseteq 2^A$ of subsets of $A$, such that (1) $\varnothing\in K$; (2) if $I\in K$ and $J\subset I$, then $J\in K$. The elements of $A$ are called \emph{vertices}, the elements of $K$ are called \emph{simplices}. The value $\dim I = |I|-1$ is called the dimension of a simplex $I$. The maximal dimension of simplices is called the dimension of $K$. A simplex $I$ is called \emph{maximal} (or \emph{a facet}), if there is no $J\in K$ which strictly contains $I$. A simplicial complex is called \emph{pure} if all facets have the same dimension. In a pure simplicial complex, a simplex $J$ is called \emph{a ridge}, if $\dim J=\dim K-1$. An element $i\in A$ is called a ghost vertex of $K$ if $\{i\}\notin K$.

If $K_1,K_2$ are simplicial complexes on the vertex sets $A_1,A_2$ respectively, then the join $K_1\ast K_2$ is a simplicial complex $\{I_1\sqcup I_2\subset A_1\sqcup A_2\mid I_1\in K_1, I_2\in K_2\}$. The full simplex on a set $A$ is the simplicial complex $\Delta_A=2^A$ of all subsets of $A$. The boundary of a simplex on a set $A$ is the simplicial complex $\dd\Delta_A=2^{[A]}\setminus\{A\}$ of all proper subsets of $A$. The ghost complex on a set $A$ is the simplicial complex $o_A=\{\varnothing\}$, in which all vertices are ghost.

\begin{con}\label{conIndepComplex}
For a multiset $\alpha=\{\alpha_1,\ldots,\alpha_r\}$ of vectors in a rational vector space $\Nq$ consider a simplicial complex $K(\alpha)$ on the vertex set $[r]=\{1,\ldots,r\}$ whose simplices are the linearly independent subsets of vectors:
\[
\{i_1,\ldots,i_l\}\in K(\alpha)\Leftrightarrow \alpha_{i_1},\ldots,\alpha_{i_l} \text{ are linearly independent}.
\]
By definition, $K(\alpha)$ is the independence complex of the linear matroid determined by the collection $\alpha$. Since $\alpha$ spans $\Nq\cong\Qo^k$, the complex $K(\alpha)$ is pure of dimension $k-1$.
\end{con}

\begin{rem}\label{remWedges}
As proved by Bj\"{o}rner~\cite{Bjorner}, the independence complex of any matroid is shellable, hence homotopically Cohen--Macaulay. This implies that the geometrical realization $|K(\alpha)|$ is homotopy equivalent to a wedge of $(k-1)$-dimensional spheres.
\end{rem}

Recall the classical notion of combinatorial topology.

\begin{defin}\label{definPseudomfd}
A pure simplicial complex $K$ is called \emph{a (closed) pseudomanifold} if any ridge is contained in exactly two facets. A pure simplicial complex $K$ is called \emph{a pseudomanifold with boundary} if any ridge is contained in one or two facets.
\end{defin}

\begin{rem}
If a ridge is contained in one facet, it is called \emph{a boundary ridge}. When we use the term pseudomanifold with boundary we assume that there exists at least one boundary ridge. So a pseudomanifold without boundary is not considered a pseudomanifold with boundary.
\end{rem}

Our proof of Theorem~\ref{thmAll} is essentially based on the next lemma. For convenience we call the assumption of item 1 in  Theorem~\ref{thmAll} \emph{the manifold case}, and the assumption of item 2 \emph{the boundary case}. Corresponding representations are called respectively \emph{representations of manifold type}, and \emph{representations of boundary type}.

\begin{prop}\label{propMfdPseudomfd}
Consider a representation of a torus $T=T^k\to\GL(V)$, and let $\alpha=\{\alpha_1,\ldots,\alpha_r\}\in\Nq$ be the defining multiset of weights. Then the following hold true.
\begin{enumerate}
  \item In the manifold case, the simplicial complex $K(\alpha)$ is a pseudomanifold.
  \item In the boundary case, the simplicial complex $K(\alpha)$ is a pseudomanifold with boundary.
\end{enumerate}
\end{prop}

It will be more convenient for us to work with homology manifolds instead of topological manifolds. For a space $Q$, the relative homology modules $H_*(Q,Q\setminus\{x\};R)$ are called \emph{the local homology modules} at a point $x\in Q$ with coefficients in an abelian group $R$. Recall that a locally compact space $Q$ is called \emph{a (closed) $d$-dimensional homology manifold} (over $R$) if, for any $x\in Q$, the local homology modules are isomorphic to those of $\Ro^d$:
\begin{equation}\label{eqHomologyModel}
H_s(Q,Q\setminus\{x\};R)\cong H_s(\Ro^d,\Ro^d\setminus\{0\};R)\begin{cases}
                                                                \cong R, & \mbox{if } s=d; \\
                                                                =0, & \mbox{otherwise}.
                                                              \end{cases}
\end{equation}
A space $Q$ is called \emph{a $d$-dimensional homology manifold with boundary}, if its local homology modules are isomorphic to those of $\Rg\times\Ro^{d-1}$: either vanish (for boundary points) or satisfy~\eqref{eqHomologyModel} (for interior points). The next statement is a direct consequence of K\"{u}nneth formula for relative homology groups.

\begin{lem}\label{lemHomMfd}
Let us fix a ring $R$ of coefficients.
\begin{enumerate}
  \item If $Q$ is a closed homology manifold, then so is $Q\times\Ro^s$.
  \item If $Q$ is a homology manifold with boundary, then so is $Q\times\Ro^s$.
  \item If $Q$ is not a homology manifold (with or without boundary), then neither is $Q\times\Ro^s$.
\end{enumerate}
\end{lem}

The next technical lemma is needed for the proof of Proposition~\ref{propMfdPseudomfd}.

\begin{lem}\label{lemCircle}
Consider a $T^1$-representation on $V\cong\Ro^{2n}$, $n\geqslant 1$, with no trivial component. Then we have an alternative.
\begin{enumerate}
  \item $n=1$, $\Co^1/T^1$ is homeomorphic to $\Rg$.
  \item $n=2$, $\Co^2/T^1$ is homeomorphic to $\Ro^3$.
  \item $n\geqslant 3$, $\Co^n/T^1$ is not a homology manifold (neither closed nor a homology manifold with boundary) over any $R$.
\end{enumerate}
\end{lem}

\begin{proof}
Item (1) is straightforward, see Example~\ref{exCompl0}). Item (2) follows from Proposition~\ref{propLocalQuotient}. The additional details concerning item (2) are provided in Section~\ref{secMonopoles}. We concentrate on item (3).

Since there is no trivial component, we have $V\cong V(\alpha_1)\oplus\cdots\oplus V(\alpha_n)$, where $\alpha_1,\ldots,\alpha_n\in\Hom(T^1,T^1)$ is a collection of nonzero integers. In the complex coordinates associated with the irreducible summands $V(\alpha_i)$, the representation takes the form
\[
t(z_1,\ldots,z_n)=(t^{\alpha_1}z_1,\ldots,t^{\alpha_n}z_n).
\]
Restricting this action to the unit sphere $S^{2n-1}=\{\sum_{i=1}^{n}|z_i|^2=1\}$ we get the weighted projective space $\CP^{n-1}(\alpha)=\CP^{n-1}(\alpha_1,\ldots,\alpha_n)$ as the orbit space. Therefore $\Co^n/T^1$ is an open cone $\Cone\CP^{n-1}(\alpha)$ with an apex denoted by $0$. We have
\[
H_j(\Cone\CP^{n-1}(\alpha),\Cone\CP^{n-1}(\alpha)\setminus\{0\};R)\cong H_{j-1}(\CP^{n-1}(\alpha);R).
\]
The weighted projective space $\CP^{n-1}(\alpha)$ has the same homology as an ordinary projective space $\CP^{n-1}$, over any $R$~\cite{Kawa}. Therefore we have a nonvanishing local homology module $H_3(\Co^n/T,(\Co^n/T)\setminus\{0\};R)\cong H_2(\CP^{n-1})\cong R$ which is an obstruction for the $(2n-1)$-dimensional space $\Co^n/T$ to be a homology manifold when $n\geqslant 3$.
\end{proof}

Now we can prove Proposition~\ref{propMfdPseudomfd} by reducing it to the circle case.

\begin{proof}
Since the trivial component of the action does not affect the statement we assume for simplicity that there is no trivial component.

Consider any ridge $J=\{j_1,\ldots,j_{k-1}\}\in K(\alpha)$. Recall that the multiset $\alpha=\{\alpha_1,\ldots,\alpha_r\}$ linearly spans $\Nq\cong\Qo^k$. Let $\Pi_J\subset\Nq$ be the rational hyperplane spanned by $\alpha_{j_1},\ldots,\alpha_{j_{k-1}}$. We partition all weights' indices into two disjoint classes: $[r]=A_J\sqcup B_J$, one for the weights lying in $\Pi_J$, and another for the weights transversal to $\Pi_J$:
\[
A_J=\{j\in[r]\mid \alpha_j\in \Pi_J\},\qquad B_J=[r]\setminus A_J.
\]
Notice that the set $B_J$ consists of all indices $i$ such that $\{\alpha_j\mid j\in \{i\}\sqcup J\}$ is linearly independent and has rank $k$. Therefore $B_J$ parameterizes the ways to complement the ridge $J$ to the facet in $K(\alpha)$. Hence
\begin{equation}\label{eqClaimBj}
  |B_J| \text{ equals the number of facets containing } J.
\end{equation}

Consider the decomposition of $V=V_A\oplus V_B$ into two summands corresponding to the partition of the weights:
\[
V_A=\bigoplus_{i\in A}V(\alpha_i)\cong \Co^{|A_J|}, \qquad  V_B=\bigoplus_{i\in B}V(\alpha_i)\cong \Co^{|B_J|}
\]
Notice that $V_A$ is the fixed point set of the 1-dimensional toric subgroup
\[
G=\Ker\prod_{j\in J}\alpha_j\colon T^k\to T^{k-1}=\Ker\prod_{j\in A_J}\alpha_j\colon T^k\to T^{|A_J|}
\]
(more precisely, we take the connected component of $1$ in these kernels to avoid disconnected groups). A general fact is described in Section~\ref{secLowerFaces}: the flats of the linear matroid $\alpha$ are in bijective correspondence with the fixed point sets of toric subgroups acting on $V$. Here we apply this correspondence to the flat $A_J$ of the matroid of weights.

Now we can summarize the idea of proof as follows. If we take a generic point $x$ in $V_A$, its tangent space decomposes as the sum of the tangent and normal components (parallel to $V_A$ and $V_B$ respectively). Then, informally, the $T^k$-action in vicinity of $x$ splits into ``the product'' of the $T^k/G$-action on the tangent component and the $G$-action on the normal component. Since $x$ is generic in $V_A$, the action of $T^k/G$ is free on the tangent component, so the orbit space of the tangent space is a manifold, and does not affect the local topology of the orbit space by Lemma~\ref{lemHomMfd}. The $G$-action on the normal component is a circle representation on $\Co^{|B_J|}$. We are in position to apply Lemma~\ref{lemCircle}. In the manifold case, this lemma implies $|B_J|=2$, and in the boundary case it implies $|B_J|=1$ which proves the required statement according to~\eqref{eqClaimBj}.

In order to justify this argument, the Slice Theorem should be applied. Let $x$ be a point in $V_A\subset V$ such that all its coordinates in this subspace are nonzero. For example, one can take
\[
x=(\underbrace{1,\ldots,1}_{A_J},\underbrace{0,\ldots,0}_{B_J}).
\]
Let $\tau_xV$, $\tau_xV_A$, and $\nu_x$ be respectively the tangent space to $V$, the tangent space to $V_A$, and the normal space of the embedding $V_A\subset V$ taken at the point $x$. Obviously, $\tau_xV=\tau_xV_A\oplus \nu_x$, $\tau_xV_A\cong \Co^{A_J}$ and $\nu_x\cong\Co^{B_J}$. The stabilizer $T_x$ of the point $x$ is the circle $G$ introduced above, so the orbit $T^kx$ is $(k-1)$-dimensional. The Slice Theorem states that the orbit $T^kx$ has a $T^k$-invariant neighborhood $U$ equivariantly diffeomorphic to
\[
T^k\times_G(\tau_xV/\tau_xT^kx)
\]
Let $[x]\in V/T^k$ denote the class of the point $x$ in the orbit space. Then $[x]$ has an open neighborhood in $V/T^k$ equal to
\begin{equation}\label{eqLocalPatch}
U/T^k\cong (T^k\times_G(\tau_xV/\tau_xT^kx))/T^k=(\tau_xV/\tau_xT^kx)/G.
\end{equation}
Notice that the whole orbit $T^kx$ lies inside $V_A$, so $\tau_xT^kx\subset\tau_xV_A$. Moreover, since $G$ is the stabilizer of $V_A$, the $G$-action on the whole subspace $\tau_xV_A$ is trivial. Therefore the $G$-action on $\tau_xV/\tau_xT^kx$ has the same weights as the $G$-action on $\tau_xV/\tau_xV_A=\nu_x$. On the other hand, the $G$-action on $\nu_x\cong \Co^{|B_J|}$ is nontrivial (its weights are the projections of $\{\alpha_j\mid j\in B\}$ under the induced map $N=\Hom(T^k,T^1)\to \Hom(G,T^1)\cong \Zo$, and these projections are nonzero by the construction of $B$). Therefore, applying Lemma~\ref{lemCircle} to the representation in~\eqref{eqLocalPatch} we see that the manifold case implies $|B_J|=2$ and the boundary case implies $|B_J|=1$ as desired.
\end{proof}

By Remark~\ref{remWedges}, any independence complex $K(\alpha)$ is homotopically equivalent to a wedge of spheres. If, moreover, it is a pseudomanifold, there is a fundamental class in the top homology group, so we necessarily have one sphere in the wedge. If $K(\alpha)$ is a pseudomanifold with boundary, then the wedge consists of zero spheres, hence in this case $|K(\alpha)|$ should be contractible. It happens that the condition of being both a matroid and a pseudomanifold puts even a stronger restriction on the combinatorics of a complex as proved by Provan and Billera in~\cite{BilProv}.

\begin{prop}[\cite{BilProv}]\label{propBillProv}\mbox{}
\begin{enumerate}
  \item If $K$ is an independence complex of a matroid and, at the same time, a closed pseudomanifold, then $K$ is isomorphic to a join of boundaries of simplices, and, probably, a ghost complex.
  \item If $K$ is an independence complex of a matroid and, at the same time, a pseudomanifold with boundary, then $K$ is isomorphic to a join of a simplex, boundaries of simplices, and, probably, a ghost complex.
\end{enumerate}
\end{prop}

\begin{rem}\label{remGhostIsComplexityOne}
Note that the ghost complex on one vertex can be formally considered as the boundary of 0-dimensional simplex. So it will not be a mistake to remove the mention of ghost simplex from the formulation of Proposition~\ref{propBillProv}.
\end{rem}

To finalize the proof of Theorem~\ref{thmAll} it remains to make a simple terminological observation.

\begin{rem}\label{remCorrespondence}
Recall that any collection of vectors (weights) $\alpha$ gives rise to the independence complex $K(\alpha)$. Properties of simplicial complexes translate to weight systems as follows.
\begin{enumerate}
  \item There is an operation of direct sum of matroids. If $\alpha\subset\Qo^k$, and $\beta\subset\Qo^l$, then the direct sum is defined $\alpha\sqcup\beta\subset\Qo^k\times\Qo^l\cong\Qo^{k+l}$, where $\alpha$ sits in the first summand, and $\beta$ sits in the second. Then $K(\alpha\sqcup\beta)=K(\alpha)\ast K(\beta)$. Vice versa, if $K(\alpha)$ splits as the join of two independence complexes, then the weights of $\alpha$ split in two groups lying in transversal vector subspaces, corresponding to the join factors. Recalling that the ambient vector spaces in our considerations are $\Nq=\Hom(T,T^1)\otimes\Qo$, it is seen that the join operation on the simplicial complexes corresponds to the direct product of torus representations.
  \item A simplex $\Delta_A$ is an independence complex of a linearly independent set in $\Nq$. This situation corresponds to representations of complexity zero, see Example~\ref{exCompl0}.
  \item A boundary of simplex $\dd\Delta_A$ is an independence complex of a weight system $\alpha_1,\ldots,\alpha_{|A|}$ where every $|A|-1$ vectors are independent, but the whole system is not. These are the weights of complexity one representations in general position by Definition~\ref{definComplOneGenPos}.
  \item Ghost vertices resemble loops in a matroid. They correspond to zero weights, in other words, the trivial component of the action. In accordance with Remark~\ref{remGhostIsComplexityOne}, a trivial torus action on $\Co$ (or $\Ro$) can be considered as a degenerate case of complexity one torus action in general position.
\end{enumerate}
\end{rem}

Theorem~\ref{thmAll} now follows from Propositions~\ref{propMfdPseudomfd} and~\ref{propBillProv} and Remark~\ref{remCorrespondence}.

\section{Torus actions}\label{secTorusActions}

\begin{con}\label{conFaceSubmanifolds}
Consider a smooth action of a torus $T$ on a connected smooth manifold $X$. If $H\subset T$ is a connected subgroup, any connected component $Y$ of the fixed point submanifold $X^H$ is called \emph{an invariant submanifold} of the action. Since $T$ is commutative, invariant submanifolds are stable under $T$-action. The dimension of the generic toric orbit on $Y$ is called \emph{the rank} of an invariant submanifold $Y$. If $Y\cap X^T\neq\varnothing$ (i.e. $Y$ contains a $T$-fixed point), then $Y$ is called \emph{a face submanifold} of the torus action.

The collection of all face submanifolds in $X$ is a poset (graded by the ranks) which we denote by $S(X)$. The poset $S(X)$ has the greatest element, the manifold $X$ itself. All minimal elements have rank $0$, these are the connected components of the fixed point set~$X^T$.

The orbit space $Y/T$ of a face submanifold $Y$ is called \emph{a face} of the action. Faces are subspaces of the orbit space $X/T$. Obviously, they are partially ordered by inclusion, and the poset of faces is naturally identified with $S(X)$. The poset $S(X)$ of faces carries a lot of useful information about the torus action as evidenced by the next examples.
\end{con}

\begin{ex}
If $X$ is a smooth complete toric variety, $S(X)$ is isomorphic to the poset of cones of its fan ordered by reversed inclusion. In particular, the Betti numbers of $X$ are determined by the combinatorics of $S(X)$ since they coincide with the $h$-numbers of the corresponding simplicial sphere. Similar statement holds for topological generalizations of toric varieties: quasitoric manifolds~\cite{DJ} and equivariantly formal torus manifolds~\cite{MasPan}.
\end{ex}

\begin{ex}
In~\cite{AyzCompl,AyzMasEquiv} the combinatorics and topology of the poset $S(X)$ was described for torus actions of complexity one in general position with isolated fixed points. In particular, it was proved in~\cite{AyzMasEquiv}, that the Betti numbers of equivariantly formal manifolds with the listed properties are determined by the poset $S(X)$.
\end{ex}

\begin{rem}
If a torus action on $X$ is equivariantly formal and has isolated fixed points, we do not know if the poset $S(X)$ determines the Betti numbers of $X$ in general.
\end{rem}

\begin{rem}
The study of general properties of the face posets of torus actions with isolated fixed points was initiated by the first author in~\cite{AyzCherep,AyzMasSolo}. Nontrivial examples of such posets related to regular semisimple Hessenberg varieties appeared in~\cite{AyzBuchGraph}.
\end{rem}

The assumption that a face submanifold should intersect the fixed point set allows to localize consideration of orbit spaces to the vicinity of fixed points. In the vicinity of fixed points the action can be linearized and reduced to the study of torus representations. Under appropriate assumptions about fixed points, we can prove a smooth version of Theorem~\ref{thmAll}. For convenience we introduce the following notion.

\begin{defin}\label{definLeontActn}
A $T$-action on a smooth manifold $X$ is called \emph{a Leontief (totally Leontief) action}, if, for any fixed point $x\in X^T$, the tangent representation $\tau_xT$ is a Leontief (resp. totally Leontief) representation.
\end{defin}

The action is called non-totally Leontief if it Leontief but not totally Leontief. Equivalently, all fixed points have Leontief tangent representations but at least one of these tangent representations is not totally Leontief.

\begin{prop}\label{propOrbitsLeontief}
Let a torus $T$ act smoothly on a connected closed smooth manifold $X$. Assume that each invariant submanifold of $X$ is a face submanifold, in other words, each invariant submanifold contains a fixed point. The the following statements hold.
\begin{enumerate}
  \item The action is totally Leontief if and only if $X/T$ is a closed topological manifold.
  \item The action is non-totally Leontief if and only if $X/T$ is a topological manifold with boundary.
  \item The action is non-Leontief if and only if $X/T$ is not a topological manifold.
\end{enumerate}
\end{prop}

\begin{proof}
The proof repeats~\cite[Thm.2.10]{AyzCompl} so we only sketch a general idea. In the vicinity of a fixed point $x$, the orbit space is homeomorphic to $\tau_xX/T$ so the statement follows from Theorem~\ref{thmAll}. If $x'$ is any other point, then $x'$ lies in a principal orbit of some invariant submanifold $Y$. Since $Y$ contains some fixed point $x$, we can continuously move $x'$ until we get in the vicinity of $x$. Since $[x]$ has a neighborhood in $X/T$ homeomorphic to an open disc (or a halfspace), the same holds for the orbit class $[x']$.
\end{proof}

\begin{rem}
The assumption that each invariant submanifold is a face submanifold may seem complicated and hard to check in practice. However, most actions automatically satisfy this property. All actions with $H^{\odd}(X)=0$ have this property as follows from~\cite[Lm.2.2]{MasPan}. In particular, equivariantly formal torus actions with isolated fixed points have the property. Algebraic torus actions on smooth projective varieties satisfy this assumption according to Bialynicki-Birula theory, see details in~\cite{AyzCompl}.
\end{rem}

\section{Faces of Leontief representations}\label{secLowerFaces}

If a torus representation $T\to\GL(V)$ is given, all invariant submanifolds of $V$, in the sense of Construction~\ref{conFaceSubmanifolds}, are $T$-invariant vector subspaces of $V$. All of them are face submanifolds, since they necessarily contain the fixed point $0\in V$. It is not very difficult to describe the combinatorial condition for a $T$-invariant vector subspace of $V$ to be a face submanifold. Let us recall the notion of flats of a linear matroid.

\begin{con}
Let $\alpha=\{\alpha_1,\ldots,\alpha_m\}$ be a linear matroid, that is a multiset of vectors in some vector space $W$. A subset $\{i_1,\ldots,i_s\}\subseteq[m]$, or the corresponding submultiset $A=\{\alpha_{i_1},\ldots,\alpha_{i_s}\}$, is called \emph{a flat} of the linear matroid if $A$ is an intersection of $\alpha$ with some vector subspace $\Pi\subset W$. The dimension of the linear span of $A$ is called the rank of a flat $A$. The flats of the matroid $\alpha$ are partially ordered by inclusion, they form a graded poset, which is called \emph{the geometric lattice} of the matroid $\alpha$ and is denoted $\Flats(\alpha)$.
\end{con}

In~\cite{AyzCherep} we observed the following

\begin{prop}
Let $V=\bigoplus_{i=1}^rV(\alpha_i)\oplus\Ro^{m-2r}$ be a representation of the torus with the weights $\alpha$. Then all face submanifolds of $V$ have the form
\[
\bigoplus_{\alpha_i\in A}V(\alpha_i)\oplus\Ro^{m-2r}
\]
where $A$ is a flat of the rational matroid of weights $\alpha\subset\Nq$. Therefore, in particular, the poset $S(V)$ is isomorphic to the geometric lattice $\Flats(\alpha)$.
\end{prop}

This statement was proved in the work~\cite{AyzCherep} under the assumption that the trivial component $\Ro^{m-2r}$ vanishes, however, the proof follows the same lines in the general case.

\begin{ex}\label{exCompl0Lattice}
If the representation of $T^n$ on $V\cong\Co^n$ is a representation of complexity zero, then $\alpha=\{\alpha_1,\ldots,\alpha_n\}$ is a basis of $\Nq\cong\Qo^n$. In this case every subset of $\alpha$ is a flat, so the poset $S(V)\cong\Flats(\alpha)$ is the boolean lattice $\B_n$.
\end{ex}

\begin{ex}\label{exCompl1Lattice}
If the representation of $T^{n-1}$ on $V=\Co^n$ is a representation of complexity zero, then we have a collection of $n$ weights $\alpha=\{\alpha_1,\ldots,\alpha_n\}$ in $\Nq\cong\Qo^{n-1}$. Every subset $A\subseteq\alpha$ is a flat unless $|A|= n-1$. Let us denote the resulting poset by $\Sp_{n-1}$:
\[
\Sp_{n-1}=\{A\subseteq[n]\mid |A|\neq n-1\}.
\]
This poset is isomorphic to the boolean lattice $\B_n$ with all coatoms removed.
\end{ex}

Recall from Definition~\ref{definLeontRepres} that the product representation of $T^d\times\prod_{i=1}^{s}T^{n_i-1}$ on $V=\Ro^l\times\Co^d\times\prod_{i=1}^{s}\Co^{n_i}$ is called a Leontief representation. We call it a Leontief representation of type $(d,\underline{n},l)=(d,\{n_1,\ldots,n_s\},l)$. Since the product of matroids induces the product of the corresponding geometric lattices, we get the following consequence of Examples~\ref{exCompl0Lattice} and~\ref{exCompl1Lattice}.

\begin{prop}\label{propFacePosetLeontief}
For a Leontief representation $V$ of type $(d,\underline{n},l)$, the face poset $S(V)$ is isomorphic to
\[
\B_d\times\prod_{i=1}^s\Sp_{n_i-1}.
\]
Let $D, N_1,\ldots,N_s$ be disjoint sets of cardinalities $d,n_1,\ldots,n_s$ respectively. Then each face submanifold of $V$ is encoded by a string $(A_0,A_1,\ldots,A_s)$, where
\[
A_0\subseteq D, \text{ and for all } i\in[s]=\{1,\ldots,s\} \text{ we have } A_i\subseteq N_i, |A_i|\neq n_i-1.
\]
\end{prop}

In toric topology, the structure of the induced torus action on the face submanifolds sometimes plays an important role. If an effective action of $T$ on $X$ has rank $k$, and $Y\subset X$ is a face submanifold of rank $l$, then the induced action of $T$ on $Y$ has noneffective kernel of dimension $k-l$. We may quotient out this noneffective kernel.

It should be noted that the class of Leontief representations is closed under taking faces.

\begin{lem}\label{lemFaceIsLeontief}
Consider a Leontief representation $V$ of type $(d,\{n_1,\ldots,n_s\},l)$, and let $U$ be a face submanifold of $V$ corresponding to the string $(A_0,A_1,\ldots,A_s)$ as in Proposition~\ref{propFacePosetLeontief}. Let $\ca{M}=\{i\in[s]\mid |A_i|=n_i\}$. Then $U$ is a Leontief representation of type $(d',\underline{n'},l)$ where
\[
\underline{n'}=\{n_i\mid i\in \ca{M}\},\text{ and } d'=|A_0|+\sum_{i\in[s]\setminus \ca{M}}|A_i|.
\]
\end{lem}

In other words, the complexity one component $A_i$, $i=1,\ldots,s$, of the string either contributes to a complexity one component (if $|A_i|=n_i$), or contributes to the complexity zero component (if $|A_i|\leqslant n_i-2$). The lemma is proved by a straightforward examination of flats in the weight matroid of a Leontief representation.

Lemma~\ref{lemFaceIsLeontief} immediately implies

\begin{cor}\label{corFaceIsLeontief}
The induced action on a face submanifold of a Leontief action is Leontief.
\end{cor}

Recall that a face of an action is the orbit space of a face submanifold. Corollary~\ref{corFaceIsLeontief}, Theorem~\ref{thmAll}, and Proposition~\ref{propOrbitsLeontief} imply the following result.

\begin{prop}\label{propMfdFaceMfd}
Let a torus $T$ act smoothly on a closed smooth manifold $X$. Assume that each invariant submanifold of $X$ is a face submanifold. If the orbit space $X/T$ is a topological manifold (either closed or with boundary) then each face of the action is also a topological manifold (either closed or with boundary).
\end{prop}

\begin{ex}
For actions of complexity zero (which are particular cases of Leontief actions), the orbit space is a manifold with boundary. All its faces are also manifolds with boundary except for the minimal elements of $S(X)$. These minimal elements are the connected components of $X^T$: these are closed manifolds\footnote{Abusing notation we identify $X^T$ with $X^T/T$}. Note that an isolated point is considered a closed manifold not a manifold with boundary.
\end{ex}

\begin{ex}
For actions of complexity one in general position, all proper face submanifolds (all except $X$ itself) have complexity $0$. The local structure of faces in the vicinity of a fixed point was described in~\cite{AyzCompl} and axiomatized in the notion of \emph{a sponge}.
\end{ex}

Finally, we make a simple observation which relates Leontief actions to the work of Cherepanov~\cite{Cher} on complexity one actions in non-general position.

\begin{lem}
Every representation of complexity one is a Leontief representation.
\end{lem}

\begin{proof}
A representation of complexity one is characterized by $n$ weights $\alpha_1,\ldots,\alpha_n$ in the $(n-1)$-dimensional vector space $\Nq\cong\Qo^{n-1}$. Hence there is a unique (up to multiplier) linear relation $c_1\alpha_1+\cdots+c_n\alpha_n=0$. The weights' subset $\{\alpha_i\mid c_i\neq0\}$ corresponds to a complexity one action in general position, while the remaining weights correspond to an action of complexity $0$.
\end{proof}

\begin{cor}
For an action of complexity one in non-general position, the orbit space is a topological manifold with boundary. All its faces are topological manifolds, either closed or with boundary.
\end{cor}

\appendix

\section{Leontief substitution systems}\label{secLeontief}

In this section we explain the term Leontief representation by making a survey about Leontief substitution systems and drawing some analogies between the combinatorial theory which appeared in mathematical economics and representations described Section~\ref{secDefinitions}. It essentially follows the results of~\cite{BilProv}, however we provide some details which may be of use for researchers in toric topology.

Let $A$ be a real matrix with $g$ rows and $f$ columns, and $b\in\Ro^g$ be a column vector. Consider the convex polyhedron $P$ determined by the system
\begin{equation}\label{eqPolyhedron}
Ax=b,\quad x\geqslant 0.
\end{equation}

\begin{defin}\label{defLeonSys}
System~\eqref{eqPolyhedron} (and the corresponding polyhedron $P$) is called \emph{a Leontief substitution system} if $b\geqslant 0$, each column of $A$ contains at most one positive entry, and $P$ is nonempty. If, moreover, the polyhedron $P$ is bounded, then it is called \emph{a totally Leontief substitution system}.
\end{defin}

\begin{rem}\label{remEconomicsExplained}
Leontief work influenced the field of mathematical economics. In particular, the general task of linear programming was, to much extent motivated by the Leontief models. Originally~\cite{Leon} Leontief systems were introduced to model the following setting.

Assume that we have $g$ commodities (``goods''), and $f$ production sites (``factories''). In a production cycle each factory consumes some commodities and either produces a new commodity, or does not produce anything at all. The production cycle of the $j$-th factory is therefore characterized by some column vector $(a_{1,j},a_{2,j},\ldots,a_{g,j})^t$ where $a_{i,j}$ is the output of $i$-th resource during a cycle. Since all resources, probably except one, are consumed by the factory rather than produced, all entries $a_{i,j}$, probably except one, are nonpositive. Then the system~\eqref{eqPolyhedron} with the $g\times f$-matrix $A=(a_{i,j})$ is a Leontief system. It solves the problem of finding the necessary amount of production cycles $x=(x_1,\ldots,x_f)^t\geqslant0$ for each factory, in order to get the prescribed vector of goods $b=(b_1,\ldots,b_g)^t\geqslant0$. The term ``substitution'' in the definition of Leontief systems refers to the fact that one resource can be produced by several factories: the production of this resource can be substituted (in the economical sense rather than in the mathematical). This means that a positive entry can appear on the same position in several columns of $A$.

The polyhedron $P$, which is the set of all solutions to~\eqref{eqPolyhedron}, may be further used in linear programming, if one needs to minimize a total cost of production given by a linear function. Therefore, the combinatorial structure of the polyhedron $P$ of a Leontief system is of particular importance: for example, it allows to estimate the complexity of the simplex-method for the optimization task.
\end{rem}

System~\eqref{eqPolyhedron} is called \emph{nondegenerate} if the polyhedron $P$ of its solutions is simple.

\begin{ex}\label{exSimplices}
The system $Ax=b$ where $b=(1,\ldots,1)^t$ and $A$ is given by
\[
\begin{tikzpicture}[decoration=brace]
\matrix (m)[matrix of math nodes, left delimiter={(}, right delimiter={)}] {
 1 & \cdots & 1 & 0 & \cdots & 0 & \cdots & 0 & \cdots & 0 & 0 & \cdots & 0 \\
   & \underline{0} &  & 1 & \cdots & 1 &  &  & \underline{0} &  & 0 &  & 0 \\
   & \vdots &  &  &  &  & \ddots &  &  &  & \vdots & & \vdots \\
   & \underline{0} &  &  & \underline{0} &  &  & 1 & \cdots & 1 & 0 & \cdots & 0\\
};
\draw[decorate,transform canvas={yshift=0.5em},thick]
(m-1-1.north) -- node[above=2pt] {$k_1$} (m-1-3.north);
\draw[decorate,transform canvas={yshift=0.5em},thick]
(m-1-4.north) -- node[above=2pt] {$k_2$} (m-1-6.north);
\draw[decorate,transform canvas={yshift=0.5em},thick]
(m-1-8.north) -- node[above=2pt] {$k_s$} (m-1-10.north);
\draw[decorate,transform canvas={yshift=0.5em},thick]
(m-1-11.north) -- node[above=2pt] {$d$} (m-1-13.north);
\end{tikzpicture}
\]
is the simplest example of a Leontief substitution system. The polyhedron of solutions in this case is the cartesian product of standard simplices and a nonnegative cone: $P=\Delta^{k_1-1}\times\cdots\times\Delta^{k_s-1}\times\Rg^d$. This system is nondegenerate.

This example have relation to torus actions as follows. Consider the exponential map $\exp\colon\Ro^f\to T^f$ given by
\[
\exp(x_1,\ldots,x_f)=(\exp(2\pi\sqrt{-1}x_1),\ldots,\exp(2\pi\sqrt{-1}x_f)).
\]
Let $\Pi$ be the solution to the homogeneous system $Ax=0$ with the matrix $A$ written above. Then $\exp(\Pi)\subset T^f$ is a subtorus given by the relations
\[
\prod_{i=1}^{k_1}t_i=1,\quad\prod_{i=k_1+1}^{k_1+k_2}t_i=1,\quad\cdots,\quad\prod_{i=k_1+\cdots+k_{s-1}+1}^{k_1+\cdots+k_s}t_i=1.
\]
Restricting the standard representation of $T^f$ on $\Co^f$ to the subtorus $\exp(\Pi)$, we obtain the product of $s$ complexity one actions in general position and the standard representation of rank $d$. So far the orbit space $\Co^f/\exp(\Pi)$ is a manifold whenever the Leontief polyhedron $P$ is bounded, otherwise $\Co^f/\exp(\Pi)$ has boundary.
\end{ex}

Surprisingly, nondegenerate totally Leontief substitution systems are combinatorially equivalent to the system described in Example~\ref{exSimplices}. The following statements are proved in~\cite{BilProv}.

\begin{prop}\label{propBillProvEconomy}
Consider a nondegenerate totally Leontief substitution system with a polyhedron $P$ of feasible solutions. If $P$ is bounded, then $P$ is combinatorially equivalent to a product of simplices.
\end{prop}

Let $P$ be a polyhedron of solutions of a system of type~\eqref{eqPolyhedron}. For this polyhedron, a simplicial nerve-complex $K_P$ is defined. The vertices of $K_P$ correspond to facets of $P$, and a collection $\{i_1,\ldots,i_s\}$ is a simplex of $K_P$ if and only the corresponding facets intersect: $\F_{i_1}\cap\ldots\cap\F_{i_s}\neq\varnothing$. Notice that a polyhedron $P$ defined by~\eqref{eqPolyhedron} does not contain a line, since $P\subset\Rg^g$. Therefore, $P$ has a vertex. If a system~\eqref{eqPolyhedron} is nondegenerate, so that $P$ is a simple polyhedron of some dimension $k$, then $\dim K_P=k-1$. The nerve-complex $K_P$ is either a simplicial sphere (if $P$ is bounded) or a simplicial disc (if $P$ is unbounded).

Unlike general polyhedra given by~\eqref{eqPolyhedron}, nerve-complexes of Leontief systems have specific combinatorics.

\begin{lem}\label{lemLeonIsMatroid}\mbox{}
\begin{enumerate}
  \item For a nondegenerate totally Leontief system, the simplicial complex $K_P$ is isomorphic to a join $\Delta_{\ca{A}}=\dd\Delta_{A_1}\ast\cdots\ast \dd\Delta_{A_s}$ of boundaries of simplices. In this case $K_P$ is a matroid complex.
  \item For a nondegenerate non-totally Leontief system, the simplicial complex $K_P$ is isomorphic to a subcomplex of $\Delta_{\ca{A}}$ which is a PL-ball of the same dimension as $\Delta_{\ca{A}}$.
\end{enumerate}
\end{lem}

Proposition~\ref{propBillProvEconomy} is a reformulation of the first part of this lemma.

\begin{rem}
Leontief (simplicial) complexes were introduced in~\cite{BilProv} as an abstraction for the nerve-complexes of nondegenerate Leontief substitution systems. We do not give the definition here, however, we notice that among all Leontief simplicial complexes, only $\Delta_{\ca{A}}=\dd\Delta_{A_1}\ast\cdots\ast \dd\Delta_{A_s}$ and $\Delta_{\ca{A}}\ast\Delta^{d-1}$ are the independent complexes of a matroid, see~\cite[Thm.3.4]{BilProv}. Therefore, if a torus representation $T\to\GL(V)$ have weights $\alpha$, then the following statements are equivalent:
\begin{enumerate}
  \item $K(\alpha)$ is a Leontief complex;
  \item $V$ is a Leontief representation.
\end{enumerate}
This explains the name proposed for such representations in the current paper. Notice that both statements above are equivalent to $V/T$ being a manifold (with or without boundary) as stated in Theorem~\ref{thmAll}.
\end{rem}

\section{Dirac monopoles and torus actions}\label{secMonopoles}

Proposition~\ref{propLocalQuotient} is well-known and extremely important for $n=2$. According to Lemma~\ref{lemTechnical}, we may restrict ourselves to a circle action on $\Co^2\cong\Ro^4$ given by
\[
T^1=\{(t_1,t_2)\in T^2\mid t_1^{c_1}t_2^{c_2}=1\} \text{ acts by } (t_1,t_2)(z_1,z_2)=(t_1z_1,t_2z_2).
\]
with $c_1,c_2$ both nonzero and coprime. This $T^1$-representation can be rewritten in a more convenient and familiar form
\begin{equation}\label{eqDim2Hopf}
t(z_1,z_2)=(t^kz_1, t^lz_2)
\end{equation}
where $k,l$ are nonzero and coprime (it is easily seen that $k=c_2$, $l=c_1$). So far we get a particular example of the circle representation described in Lemma~\ref{lemCircle}. Here we fix the orientation of $\Ro^4$ (and the irreducible summands) compatible with the chosen complex structure.

Restricting~\eqref{eqDim2Hopf} to the unit sphere $S^3\subset\Co^2$ and taking the quotient by $T^1$ we get the weighted projective line $\CP^1(k,l)$. The latter is a (real) 2-dimensional orbifold with two isolated singularities having the isotropy groups $\Zo_k$ and $\Zo_l$. The weighted projective line $\CP^1(k,l)$ is also the quotient of $\Co^2\setminus\{0\}$ by the action of the algebraical torus $\Co^\times$ given by the same formula~\eqref{eqDim2Hopf}, so it is also an algebraic variety. The underlying topological space of $\CP^1(k,l)$ is homeomorphic to $S^2$. It is also isomorphic to $\CP^1$ as a variety. We refer to~\cite{BFNR} as a good survey of the topology of weighted projective spaces. The quotient map takes the form
\begin{equation}\label{eqGenHopf}
p_{k,l}\colon S^3\to S^3/T^1=\CP^1(k,l)\cong S^2.
\end{equation}
Of particular importance are the cases $(k,l)=(1,1)$ (the Hopf bundle), and $(k,l)=(1,-1)$ (inverse Hopf bundle).

\begin{rem}
The Hopf bundle $p_{1,1}$ is the generator of $\pi_3(S^2)\cong\Zo$. In this homotopy group, the following identity holds $[p_{k,l}]=kl[p_{1,1}]\in\pi_3(S^2)$. This can be seen by composing $p_{1,1}$ with the map $S^3\to S^3$, $(z_1,z_2)\mapsto (z_1^k,z_2^l)$ having degree $kl$.
\end{rem}

Taking the open cone of the map $p_{k,l}$, we get the map
\[
\Cone p_{k,l}\colon\Ro^4=\Cone S^3\to \Ro^4/T^1=\Cone\CP^1(k,l)\cong \Ro^3,
\]
which, in particular, justifies Proposition~\ref{propLocalQuotient} for $n=2$.

The maps $\Cone p_{1,1}, \Cone p_{1,-1}\colon \Ro^4\to\Ro^3$ serve as the local topological models of Dirac magnetic monopoles in Kaluza--Klein theory. We give a brief overview of this theory below, and refer to~\cite[p.5]{AtBer} and references therein explaining the importance of Hopf bundles in theoretical physics.

\begin{con}\label{conKaluzaKlein}
Consider a smooth $T^1$-action on an orientable 4-manifold $X$ such that all its stabilizer subgroups are connected, and fixed points of the action are isolated. Then the orbit space $Q=X/T$ is an orientable (topological) 3-manifold. The projection map $p\colon X\to Q$ is called \emph{(a simplified topological) Kaluza--Klein model}. The fixed points of the action are called \emph{magnetic monopoles}.

Let $x\in X^{T^1}$ be a fixed point. The tangent representation $\tau_xX$ is a complexity one $T^1$-representation in general position with some nonzero weights $k,l\in\Hom(T^1,T^1)$. Since $\tau_xX$ is oriented, the weights are defined up to simultaneous change of sign. If either $|k|>1$ or $|l|>1$, the representation (and hence the action on $X$) has disconnected stabilizers, see details in~\cite{AyzCompl}. Therefore we either have $(k,l)=(1,1)$ or $(1,-1)$, meaning that each tangent representation is the cone over either the Hopf bundle or the inverse Hopf bundle. We say that the magnetic monopole $x$ has charge $+1$ in the first case, and $-1$ in the second case.
\end{con}

\begin{rem}
The physical motivation for this construction goes as follows. The 3-manifold $Q$ is interpreted as a physical space, while the acting torus $T^1=U(1)$ is the gauge group of electromagnetism. Away from magnetic monopoles, the $T^1$-action is free, so it gives rise to a principal $T^1$-bundle over the physical space, resulting in a gauge theory. Historically, Kaluza--Klein model was a precursor of the more general Yang--Mills theory.

The quantum theory of magnetic monopole proposed by Dirac~\cite{Dirac}, being reformulated in topological terms, is based on the observation that there exist nontrivial $T^1$-bundles over $\Ro^3\setminus\{0\}\simeq S^2$, where $0$ is the monopole. Principal $T^1$-bundles over $S^2$ are classified by the homotopy classes
\[
[S^2,BT^1]=[S^2,K(\Zo,2)]\cong H^2(S^2;\Zo)\cong\Zo
\]
therefore the charge gets quantized. If the charge is $+1$ or $-1$, one gets Hopf or inverse Hopf bundle respectively. In this case it is possible to compactify both the base $\Ro^3\setminus\{0\}$ and the total space of the fibration and get a well-defined map $\Ro^4\to\Ro^3$ which is either $\Cone p_{1,1}$ or $\Cone p_{1,-1}$. This observation elegantly embeds Dirac quantum monopoles into Kaluza--Klein model, at least on the topological level.

Notice, however, that this observation is not suitable for Dirac monopoles with charges $q$ different from $\pm1$. If $|q|\geqslant 2$, the total space of the corresponding $S^1$-fibration on $S^2$ is homeomorphic to the lens space $L(q;1)$, hence compactifying it at the monopole point produces a singularity of type $\Cone L(q;1)$ in the total space. In this case, $X$ is not a manifold anymore.
\end{rem}

In the terminology of Construction~\ref{conKaluzaKlein}, the following statement holds.

\begin{prop}\label{propSumCharges}
In a closed Kaluza--Klein model the charges of all monopoles sum to zero.
\end{prop}

The more general statement is proved in the seminal works of Fintushel~\cite{Fint1, Fint2} who classified circle actions on 4-manifolds in terms of their orbit data. There is also a resemblance between global topological properties of $X$ and $Q$.

\begin{prop}\label{propSpheres}
Assume that Kaluza--Klein model $X\mapsto X/T^1=Q$ is closed and there is at least one monopole. Then
\begin{enumerate}
  \item $\pi_1(X)=1$ if and only if $Q$ is homeomorphic to $S^3$;
  \item the $T^1$-action on $X$ is equivariantly formal if and only if $Q$ is a homology 3-sphere.
\end{enumerate}
\end{prop}

Item (1) follows from~\cite{Fint1} and Poincar\'{e} conjecture. Item (2) is the homological version of the first item: under the assumption of isolated fixed points, equivariant formality is equivalent to the condition $H^{\odd}(X;\Zo)=0$. For oriented 4-folds this condition further simplifies to $H_1(X;\Zo)=0$. This homological statement was proved in~\cite{AyzMasEquiv} in a more general context of complexity one actions in general position. In higher dimensions we still have a correspondence between equivariant formality of a manifold, and the fact that its orbit space is a homology sphere.

\begin{rem}
It should be noted that the classical Kaluza--Klein theory considers $T^1$-bundle not over a 3-dimensional manifold, but over 4-dimensional curved space-time, so the whole theory is 5-dimensional. The aim of this theory is to incorporate both Maxwell equations and Einstein field equations uniformly as the Euler--Lagrange equations of a certain functional defined over a $T^1$-bundle on a space-time.

With time added into the model, magnetic monopoles become world lines, they are represented by 1-dimensional curves. In this case, the local model of the monopole is a circle action on $\Ro^5$ which is the product of complexity one representation in general position on $\Co^4$, and the trivial action on $\Ro^1$.

Globally these world lines may be treated as the components of the fixed point $X^{T^1}$ for a $T^1$-action on a 5-manifold $X$. Assume that the worldline of the monopole is oriented. Then the trivial component of the tangent representation gets oriented. Then the nontrivial transversal component becomes canonically oriented as well, so there is still a difference between Hopf and its inverse, and we can assign the charge $\pm1$ to the worldline depending on which one is the case. Changing an orientation of the worldline switches its charge.

Usually worldlines are assumed timelike, so we can choose their orientation canonically to agree with the causality in the ambient space-time $Y=X/T^1$. However, allowing the curves $X^{T^1}$ to be spacelike at some points, the situation when a couple of oppositely charged monopoles is born or dies becomes naturally incorporated into the model, see Fig.~\ref{figBirthDeath}.
\end{rem}

\begin{figure}
  \centering
  \includegraphics[scale=0.3]{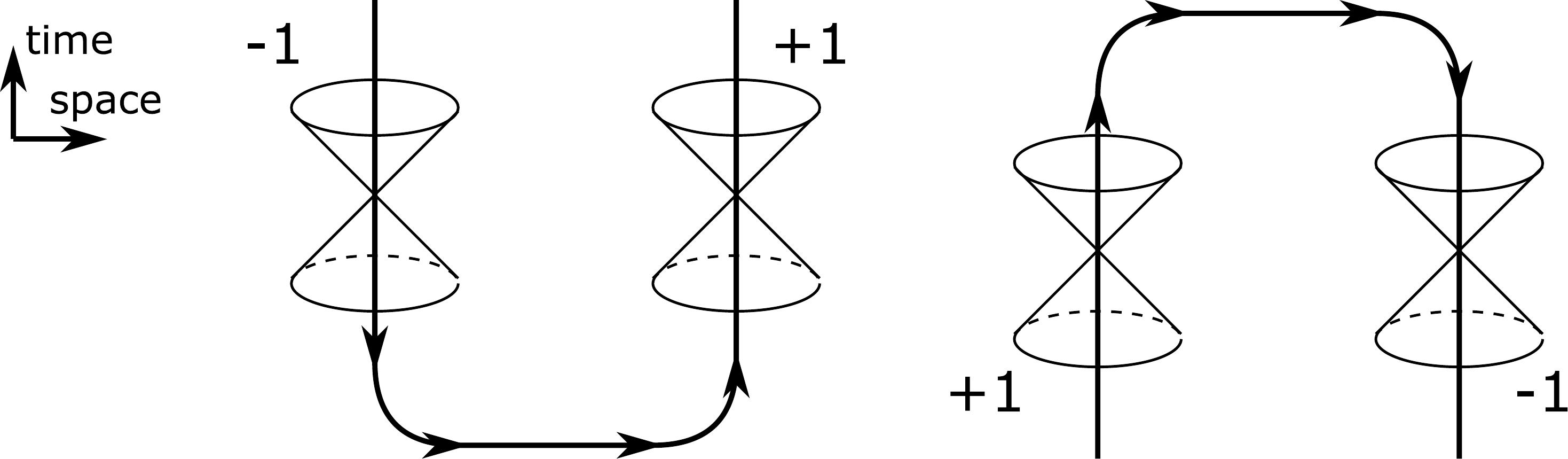}
  \caption{Birth (left) and death (right) of two oppositely charged monopoles in 5-dimensional Kaluza--Klein model. The world lines of monopoles are the connected components of a $T^1$-action on a 5-fold.}\label{figBirthDeath}
\end{figure}

The previous remark motivates further topological study of complexity one $T^1$-actions in general position on $5$-folds. In particular, it may be instructive to find the correct analogue of Proposition~\ref{propSpheres} for the 5-dimensional Kaluza--Klein theory with monopoles.

\begin{rem}
The important feature of the Kaluza--Klein model (either 4- or 5-dimensional) is that it is a torus action whose orbit space is a manifold without boundary, interpreted as the observed universe. If one restricts to the models where the gauge group is a compact torus, the analogue of the Kaluza--Klein model should be a smooth torus action on a manifold, whose orbit space is also a manifold. According to Proposition~\ref{propOrbitsLeontief}, such actions are precisely totally Leontief actions (under some assumptions on the fixed points which can be weakened in a natural way). Therefore, totally Leontief torus action give the broadest class of topological models generalizing Kaluza--Klein model.
\end{rem}

\end{document}